\documentclass[11pt]{article}

\usepackage{amsmath,amssymb,amsthm}
\usepackage{graphicx}
\usepackage{subfigure}
\usepackage{caption}

\usepackage{dsfont}
\usepackage[top=1in , bottom =1in , left=1in , right =1in]{geometry}

\makeatletter
\newcommand\figcaption{\def\@captype{figure}\caption}
\newcommand\tabcaption{\def\@captype{table}\caption}
\makeatother

\newtheorem{thm}{Theorem}[section]

\newtheorem{lem}[thm]{Lemma}

\newtheorem{cor}[thm]{Corollary}

\theoremstyle{definition}

\title{\bf Orthogonally connected sets} %
\author{Xuemei He, Xiaotian Song, Liping Yuan\thanks{Corresponding author.},  Tudor Zamfirescu}
\date{}

\begin{document}\large
\vskip -8cm \maketitle \thispagestyle{plain}
\textbf{Abstract.} In this paper, we further investigate the orthogonally connected sets and establish necessary and sufficient conditions for a set to be staircase connected.

\textbf{Keywords:} Orthogonally connected sets, staircase connected sets,  convex bodies, unimodal functions.

\textbf{Mathematics Subject Classification:} 52B10; 53A05; 54B15.

\section{Introduction}

In 1982, Montuno and Fournier \rm\cite{ftc} introduced the notion of orthogonal convexity in orthogonal polygons. In 1983, Nicholl, Lee, Liao and Wong \rm\cite{otc} worked with the orthogonal convexity in orthogonal polygons, and gave the definition of a staircase.
Ottmann, Soisalon-Soininen and Wood \rm\cite{otd} generalized the orthogonal convexity to any sets and defined orthogonal connectedness and staircase connectedness.
In 2007, Magazanik and Perles  \rm\cite{scs} studied staircase stars and kernels of staircase connected sets.

The study of staircase connectedness plays a significant role in graph theory, digital image processing and VLSI circuit layout design.
Breen did important research related to the staircases.
She considered a finite family of  boxes $\mathcal{C}$ in $\mathds{R}^d$ whose intersection graph is a tree  \rm\cite{ccs}.
She also studied
 a finite family of  boxes $\mathcal{C}$ in $\mathds{R}^d$ whose intersection graph is a block graph \rm\cite{scr}.
Besides, Breen investigated Helly-type theorems \rm\cite{ios}, and established analogues of the Radon and Carath$\mathrm{\acute{e}}$odory theorems for staircase connected  sets in $\mathds{R}^d$ \rm\cite{rad}. She also obtained an analogue of  Tverberg's theorem, using staircase connected sets instead  of convex hulls \rm\cite{tve}.

In this paper, we further investigate the orthogonally connected sets and establish necessary and sufficient conditions for a set to be staircase connected.

\section{Definitions and notation}
For distinct points $x,y\in \mathds{R}^d$, let $xy$ denote the closed line-segment from $x$ to $y$, $[xy\rangle$ the half-line from $x$ through $y$, $\overline{xy}$ the line through $x$ and $y$ and $]x,y[= xy \backslash \left\{x,y\right\}$. Always $d\geq2$.

For $M\subset \mathds{R}^d$, we denote by $\mathrm{conv}M$ its convex hull, by $\complement M$ its complement, by $\overline{M}$ its affine hull, by $\mathrm{int}M$ its interior, by $\mathrm{bd}M$ its boundary and by $\mathrm{cl}M$ its closure.

 Let $M_{1},M_{2}\subset \mathds{R}^{d}$.
If   $\overline{M_{1}}$  and  $\overline{M_{2}}$ are orthogonal, we say that  $M_{1}$ is \emph{orthogonal} to $M_{2}$, using the notation  $M_{1}\perp M_{2}$. If $\overline{M_{1}}$  and  $\overline{M_{2}}$ are parallel, we say that  $M_{1}$ is \emph{parallel} to $M_{2}$, and write  $M_{1}\parallel  M_{2}$.

Let $x\in \mathds{R}^{d}$ be a point, $M$ a subset of $\mathds{R}^{d}$, and $L$ an affine subspace of $\mathds{R}^{d}$. Affine subspaces are always supposed to have dimension  at least $1$.
Denote by $\pi_{L}(x)$ and $\pi_{L}(M) $ the orthogonal projections of $x$ and $M$ onto $L$.

The $n$-dimensional Hausdorff measure will be denoted by $\mu_{n}$.

Call a polygonal path $P=[a_0,a_1,\cdot\cdot\cdot,a_n]=a_0a_1\cup a_1a_2\cup \cdot\cdot\cdot \cup a_{n-1}a_n$ in $\mathds{R}^d$ \emph{orthogonal} if every edge of $P$ is parallel to one of the coordinate axes. An orthogonal path $P$ is a $staircase$, if, along the path, all its parallel edges point in the same direction.

The set $S\subset \mathds{R}^d$ is called $orthogonally~connected$
($staircase~connected$) if every two points $p,q\in{S}$ can be joined within $S$ by an orthogonal path (staircase).\

The set $S\subset \mathds{R}^2$ is $horizontally~convex$ ($vertically~convex$) if $S$ includes every horizontal (vertical) line-segment with endpoints in $S$. If $S\subset \mathds{R}^d$ includes every line-segment parallel to some coordinate axis and having endpoints in $S$, then we say that $S$ is $orthogonally~convex$.\

For any point $p\in \mathds{R}^{2}$, denote by $ X_p$ (resp. $Y_p$) the horizontal line (resp. vertical line) through $p$.

The $horizontal~width$ $\mathrm{hw}_x( S )$ of $S\subset \mathds{R}^2$ at $x\in S$ is $\mu_{1}( L)$, where $L$ is the component of $S\cap X_x$ containing $x$. 
The $vertical~width$ $\mathrm{vw}_x( S )$ is defined analogously.

\section{Orthogonally connected sets}

In \cite{SCS}, the authors introduced the staircase distance between two points  which is  the smallest number of edges of a staircase connected them.
Now, the \emph{$s$-distance} $d(p,q)$ from $p$ to $q$ in the orthogonally connected set $M \subset \mathds{R}^{2}$ is the minimal number of edges of an orthogonal path from $p$ to  $q$ included in $M$.  
$\mathrm{diam}_s(M)=\sup\limits_{p,q\in M} d(p,q)$.
\begin{thm}
If  $M\subset \mathds{R}^{2}$ is a convex set different from a strip, then $\complement M $ is orthogonally connected and $ \mathrm{diam}_s(\complement M)\leq 4$.
\end{thm}

\begin{proof}
Choose $p,q\in \complement M$.
Because  $M$ is not a strip, $ \complement M$ is connected.
If $pq\subset \complement M$, then $ p\pi_{X_p}(q)\cup \pi_{X_p}(q)q$ or $  p\pi_{Y_p}(q)\cup \pi_{Y_p}(q)q$ lies in $ \complement M$ and $d(p,q) \leq 2$.

Now, suppose $pq\cap  M \neq \emptyset$. Because $M$ is convex, for any $w\in \complement M$, there is a line  $L\ni w $ such that  one component of $ \mathds{R}^{2}\setminus L$ is disjoint from $ M $.
Consequently, there are  in $ \complement M$ four half-lines, $X'_p \subset X_p,Y'_p \subset Y_p $  starting at $p$, and  $X'_q \subset X_q,Y'_q \subset Y_q $ starting at $q$.

 \begin{figure}[h]
\centering
\includegraphics[width=0.6\textwidth]{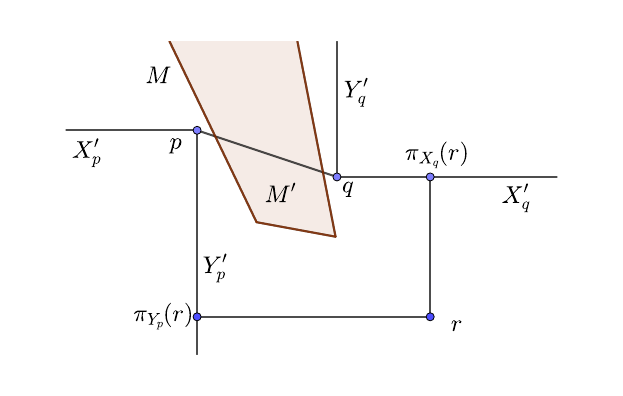}\\
\caption{}
\label{thm21}
\end{figure}

If there are two half-lines  among $ X'_p,Y'_p, X'_q,Y'_q $
pointing in the same direction, say $X'_p,X'_q$,   and there is no point $w\in  X'_p$ such that $ Y_w\subset \complement M$,
then  $M $ is unbounded and its recession cone is the positive or negative  $x$-axis $ X'_{\bf 0}$.
Without loss of generality, suppose $\pi_{X_q}(p)\in X'_q$.
If $ M\cap p\pi_{X_q}(p)= \emptyset$, $ p\pi_{X_q}(p)\cup \pi_{X_q}(p)q\subset \complement M$ is an orthogonal path and $d(p,q) \leq 2$.
If  $ M\cap p\pi_{X_q}(p)\neq \emptyset$, since the recession cone of $M$ is $ X'_{\bf 0}$,  $M\subset \mathrm{int}\mathrm{conv}(X_p \cup X_q)$.
So,  $ X_p,X_q\subset  \complement M$.
Since $\complement M$ is connected  and   $X_p\cup X_q \subset \complement M$,  there exists a point $ s\in X_p\setminus  X'_p$ such that $ Y_s\subset \complement M$. Then,  $ ps\cup s\pi_{X_q}(s) \cup \pi_{X_q}(s)q$ is an orthogonal path in $ \complement M$, and $d(p,q) \leq 3$.

If $X'_p,X'_q$ point in  the same direction and there exists a point $t\in  X'_p$ such that  $Y_t \subset \complement M$,
then $ pt\cup t\pi_{X_q}(t) \cup \pi_{X_q}(t)q$ is an orthogonal path in $ \complement M$, and $d(p,q) \leq 3$.

Suppose that $ X'_p,Y'_p, X'_q,Y'_q $ point in  different directions.
Let $M'$ be a bounded component of $M\setminus pq$. Assume without loss of generality that we are in the situation illustrated in Figure \ref{thm21}. Then there obviously exists a point $r $ lying on the same side of  $\overline{pq}$  with $M'$ such that $ \pi_{X_q}(r)\in X'_q, \pi_{Y_p}(r)\in Y'_p$  and  $ r\pi_{X_q}(r) \cup r\pi_{Y_p}(r)\subset \complement M$. Then $ p\pi_{Y_p}(r)\cup r\pi_{Y_p}(r)\cup r\pi_{X_q}(r)\cup q\pi_{X_q}(r)\subset \complement M$ and $\mathrm{diam}_s(\complement M)\leq 4$.
\end{proof}

\begin{lem}\label{dulem} {\rm\cite{du}}
Every open connected set is orthogonally connected.
\end{lem}

\begin{thm}\label{orccon2}
 If $C \subset \mathds{R}^{2}$ is a compact  orthogonally connected set and $K\subset \mathrm{int}C$ is a convex body, then $ C\setminus \mathrm{int}K$ is orthogonally connected.
\end{thm}
\begin{proof}
Let $M=C\setminus \mathrm{int}K$.   Since $K$ is  a compact convex set inside $\mathrm{int}C$,  $ M$ is connected.  Choose any two points $a,b\in M$.

If $a,b \in \mathrm{int}M$, then by   Lemma  \ref{dulem},
they are joined by an orthogonal path inside $ \mathrm{int}M$.

If at least one of the two points belongs to  $\mathrm{bd}M$, say $ a$, then,
 for   $a\in  \mathrm{bd}C$, there is an orthogonal path in $M$ joining $a $ with some point in $ \mathrm{int}M$
  since $C$ is orthogonally connected;
for  $a\in  \mathrm{bd}K$, there exists a line-segment  $aa'\subset M$, parallel to any  coordinate axis, with $a'\in \mathrm{int}M$.  
Hence, there is an orthogonal path $P_a$ joining $a\in \mathrm{bd}M$ and some interior point $a'\in \mathrm{int}M$. Similarly, we can find an orthogonal path $P_b$ joining $b\in M$ and some interior point $b'\in \mathrm{int}M$.  The union of $P_a\cup P_b$  with an orthogonal path between $a', b'$ provides an orthogonal path from $a$ to $b$.
\end{proof}

 Using the same method, we  get the following result.

\begin{thm}\label{orccon}
If  $\mathcal{K}=\{ K_i:i=1,2,\ldots,n\}$  is a family of pairwise disjoint convex bodies in  $\mathds{R}^{d}$, then  $\bigcap\limits_{i=1}^{n} \complement (\mathrm{int}K_i)$ is orthogonally connected.
\end{thm}
\begin{proof}
Let $M=\bigcap\limits_{i=1}^{n}  \complement (\mathrm{int}K_i)$.
Since  $ K_i\cap K_j =\emptyset$, for any distinct $i,j$, $M$ is connected and,  for any point $p\in M$, there is a line-segment  in $M$ starting at $p$, parallel to some  coordinate axis and meeting $\mathrm{int}M $.
Therefore, by  Lemma  \ref{dulem}, $ M$ is  orthogonally connected.
\end{proof}
\begin{figure}[h]
\centering
\includegraphics[width=0.4\textwidth]{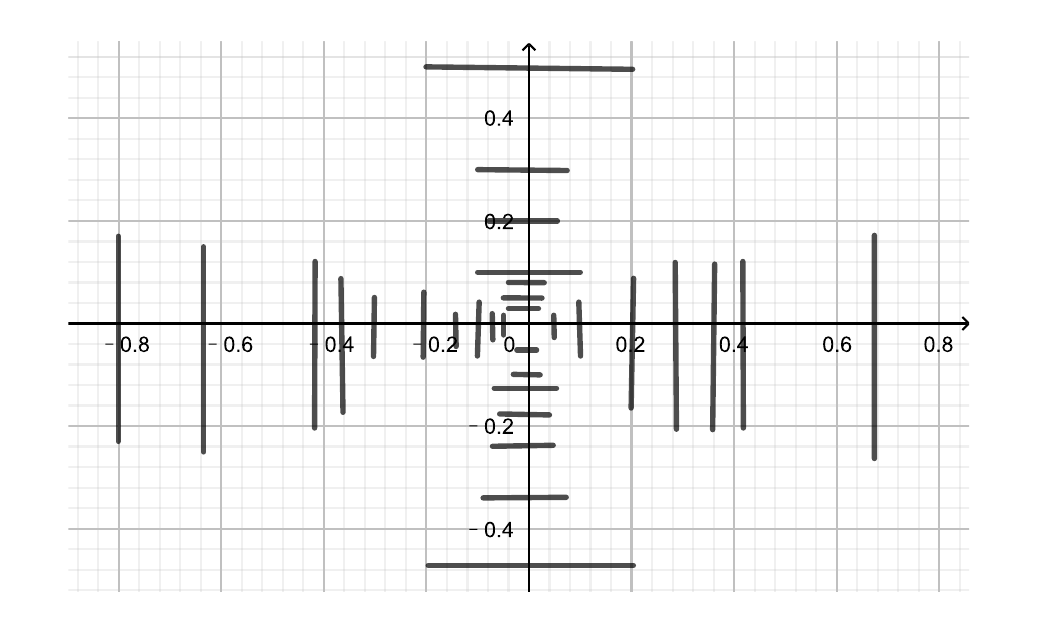}\\
\caption{}
\label{infifan}
\end{figure}

Theorem \ref{orccon} is not valid for  infinite families.
 Let  $\mathcal{K}$ be an infinite family of convex sets in  $\mathds{R}^{2}$  constructed as follows. See Figure \ref{infifan}.
Let $\{L_k\}_{k=1}^\infty$  be  a sequence of line-segments   converging to the origin $ {\bf 0}$ with $X_{\bf 0}\cap L_k \neq \emptyset$ and $ L_k\perp X_{ \bf 0}$.
For each $k$ take a parallel body $L'_k$ of $L_k$, such that all of them be pairwise disjoint. Consider an analogous sequence $\{ J_k\}_{k=1}^\infty$  with $Y_{ \bf 0}\cap J_k \neq \emptyset$ and $ J_k\perp Y_{ \bf 0}$. Arrange that $ (\pm L_i)\cap (\pm J_j)=\emptyset$, for all $i,j$. Also, consider the analogous parallel bodies $J'_k$.
  Now, take $\mathcal{K}=\{L_i'\}_{i=1}^\infty \cup \{-L_i'\}_{i=1}^\infty\cup  \{J_i'\}_{i=1}^\infty \cup \{-J_i'\}_{i=1}^\infty $.
Then ${\bf 0}\in \bigcap\limits_{K\in \mathcal{K}} \complement (\mathrm{int}K)$ and there is no orthogonal path starting at $ {\bf 0}$ in $\bigcap\limits_{K\in \mathcal{K}} \complement (\mathrm{int}K)$.

\begin{thm}\label{orcylin}
If $K \subset \mathds{R}^{2}$ is orthogonally connected, then  any right cylinder in $\mathds{R}^{3}$ based on $K$, as well as its  boundary,   are orthogonally connected.
\end{thm}
We leave the easy  proof to the reader.
\section{Staircase connectedness in the plane}

The following characterization  was obtained by Magazanik and Perles in \rm\cite{SCS}.

\begin{lem}\label{lem0} 
A set $S\subset \mathds{R}^2$ is staircase connected if and only if $S$ is orthogonally connected and orthogonally convex.
\end{lem}

\begin{thm}\label{thm20}
If $S\subset \mathds{R}^2$ is staircase connected, then $\complement S$ has no bounded connected component.
\end{thm}
\begin{proof}
Suppose that $\complement S$ has a bounded component $C$. For a point $a\in C$, let $D=X_a\cap S$.
 We have two distinct points $b$ and $c$ in $D$, such that $a\in bc$.
The boundedness of $C$ guarantees the existence of the points $b$ and $c$.
Since $bc \not\subset S $, $S$ is not orthogonally convex and therefore not staircase connected, by Lemma \ref{lem0}.
\end{proof}

\begin{thm}\label{thm21}
The set $S\subset \mathds{R}^2$ with connected interior, satisfying $S= \mathrm{cl~int} S$, is staircase connected, if and only if $S$ is orthogonally convex and, for no point $x\in S$, $\mathrm{hw}_x( S ) = \mathrm{vw}_x( S ) = 0$.\
\end{thm}

\begin{proof}
Suppose that $S\subset \mathds{R}^2$ satisfying the initial conditions of the theorem is staircase connected.
Then $S$ is orthogonally connected and orthogonally convex, by Lemma \ref{lem0}.
Therefore, for any point $x\in S$ and any other point $y\in S$, there is an orthogonal path in $S$ connecting $x$ and $y$. So, $\mathrm{hw}_{x}(S)\neq0 $ or $ \mathrm{vw}_{x}(S)\neq0$.

Now, let us prove the inverse implication.
Let $a,b\in S$ and suppose without loss of generality that the coordinates of $b$ are not smaller than those of $a$. We start at $a$ an orthogonal path $P$. It can be started, because $\mathrm{hw}_a( S ) \neq{0} $ or $\mathrm{vw}_a( S ) \neq{0} $.
We wish to come closer to $b$, horizontally or vertically.
If this is impossible, then there exist a horizontal line-segment $au$ and a vertical one $av$, with $au\backslash\left\{a\right\}$ and $av\backslash\left\{a\right\}$ disjoint from $S$, such that $b\in \mathrm{conv}([au\rangle \cup [av\rangle)$. See Figure~\ref{figau}. Since $S$ is orthogonally convex, $([au\rangle \cup [av\rangle) \cap S=\left\{a\right\}$.
\begin{figure}[htbp]
\centering
\includegraphics[width=0.35\textwidth]{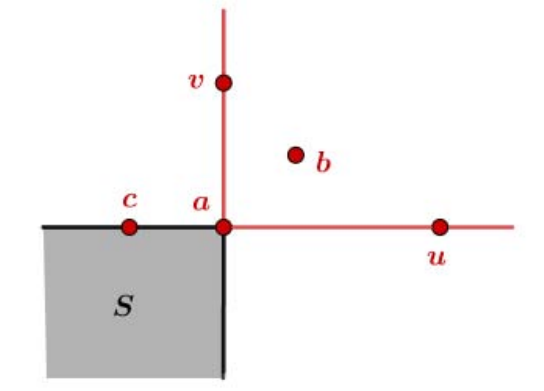}
\caption{$b\in \mathrm{conv}([au\rangle \cup [av\rangle)$.}
   \label{figau}           
\end{figure}
But there exists a horizontal or vertical line-segment $ac\subset S$. Since $S=\mathrm{cl~int}S$, there exist points of $\mathrm{int}S$ close to $b$ and points of $\mathrm{int}S$ close to $c$. Such points are separated by $[au\rangle \cup [av\rangle$,  contradicting the connectedness of $\mathrm{int}S$.
Hence we can come closer to $b$, say horizontally, until
we reach $Y_b\ni b$ or as long as we remain in $S$, up to some point $a_1$. If $a_1\in Y_b$, then $[a,a_1,b]$ is the desired staircase. See Figure~\ref{figa1}.
\begin{figure}[htbp]
\begin{minipage}[t]{0.45\linewidth}
\centering
\includegraphics[width=1.02\textwidth]{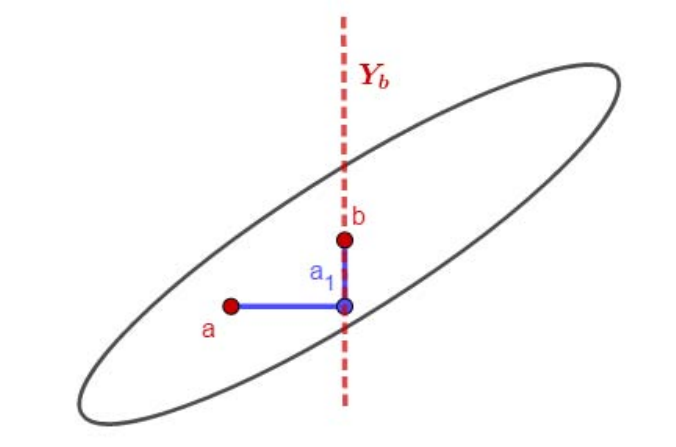}
\caption{$a_1\in Y_b$.}
   \label{figa1}           
\end{minipage}%
\hfil
\begin{minipage}[t]{0.5\linewidth}
\centering
\includegraphics[width=0.9\textwidth]{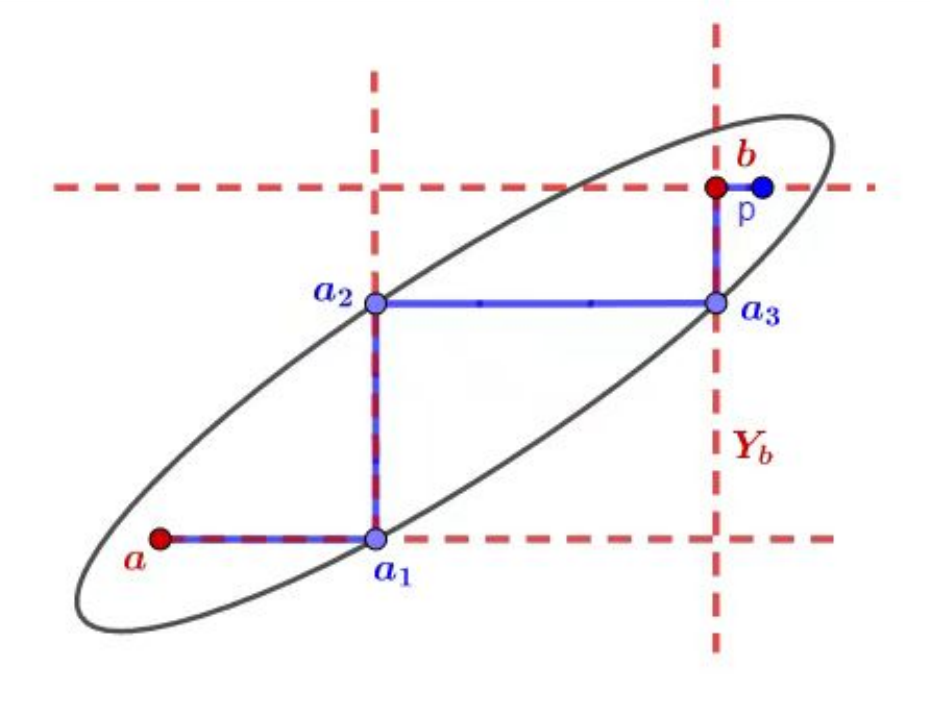}
\caption{$[a,a_1,a_2,a_3,b]$ is a staircase.}
   \label{figseq}
\end{minipage}
\end{figure}
If $a_1\in \mathrm{bd}S$, we repeat the same procedure with $a_1$ instead of $a$, and so on, until either we reach $b$, see Figure~\ref{figseq}, or we obtain a sequence of points $a_1,a_2,a_3,\cdot\cdot\cdot\in \mathrm{bd}S$, see Figure~\ref{figjh21}.
In the second case, since $S$ is compact, $\left\{a_n\right\}^\infty_{n=1}$ converges to some point $c\in S$.
Obviously, $\|a_n-a_{n+1}\|\rightarrow0$.
We have $\mathrm{hw}_c( S ) \neq{0} $ or $\mathrm{vw}_c( S ) \neq{0} $, say the first. The orthogonal convexity of $S$ together with the existence of the points $a_{2n}\in \mathrm{bd}S$ implies the non-existence of any line-segment $cs'\subset S\cap X_c$ pointing to the left. Consequently, there exists a horizontal line-segment $cs\subset S$ pointing to the right.
Similarly, there is no line-segment $cs''\subset S\cap Y_c$ pointing down.
Thus, points of $\mathrm{int} S$ close to $s$ are separated by $[cs'\rangle \cup [cs''\rangle$ from points of $\mathrm{int}S$ close to $a$.
This contradicts the hypothesis.
\end{proof}
\begin{figure}[htbp]
\centering
\includegraphics[width=0.46\textwidth]{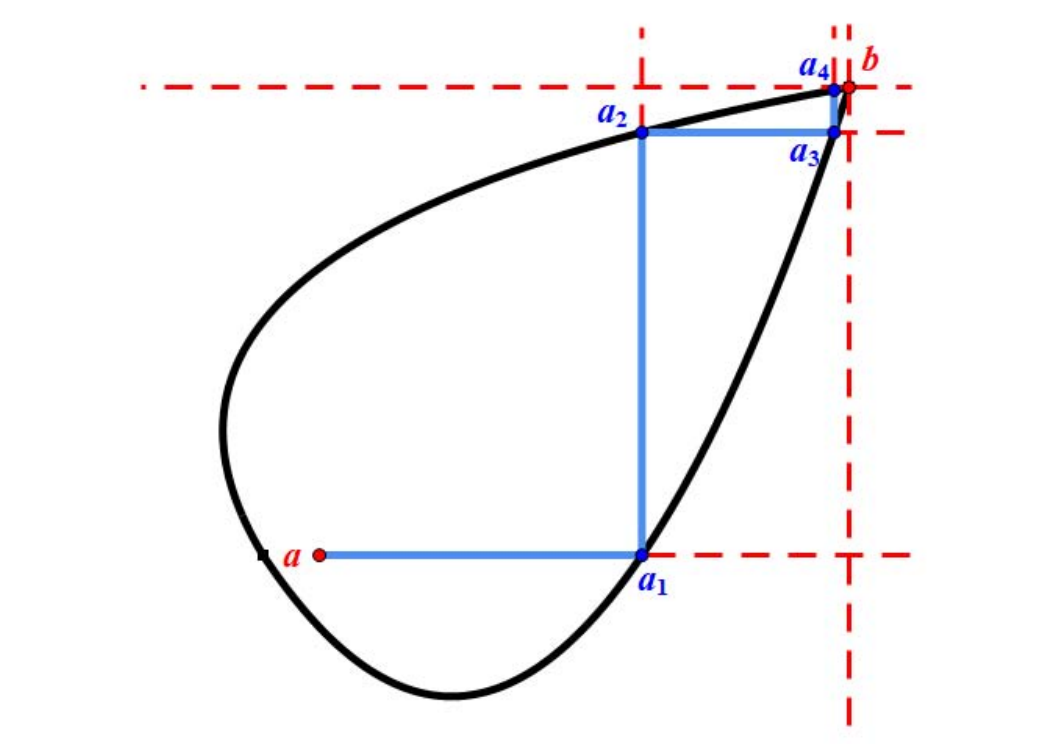}
\caption{ }  
   \label{figjh21}           
\end{figure}

Let $C\subset \mathds{R}^d$ be a continuum. A point $x\in C$ is called a $cut~point$ of $C$ if $C\setminus\left\{x\right\}$ is disconnected. In that case, the closure of any component of $C\setminus\left\{x\right\}$ will be called a \emph{piece} of $C$.
See Figure~\ref{fig4p}.
\begin{figure}[htbp]
\centering
\includegraphics[width=0.34\textwidth]{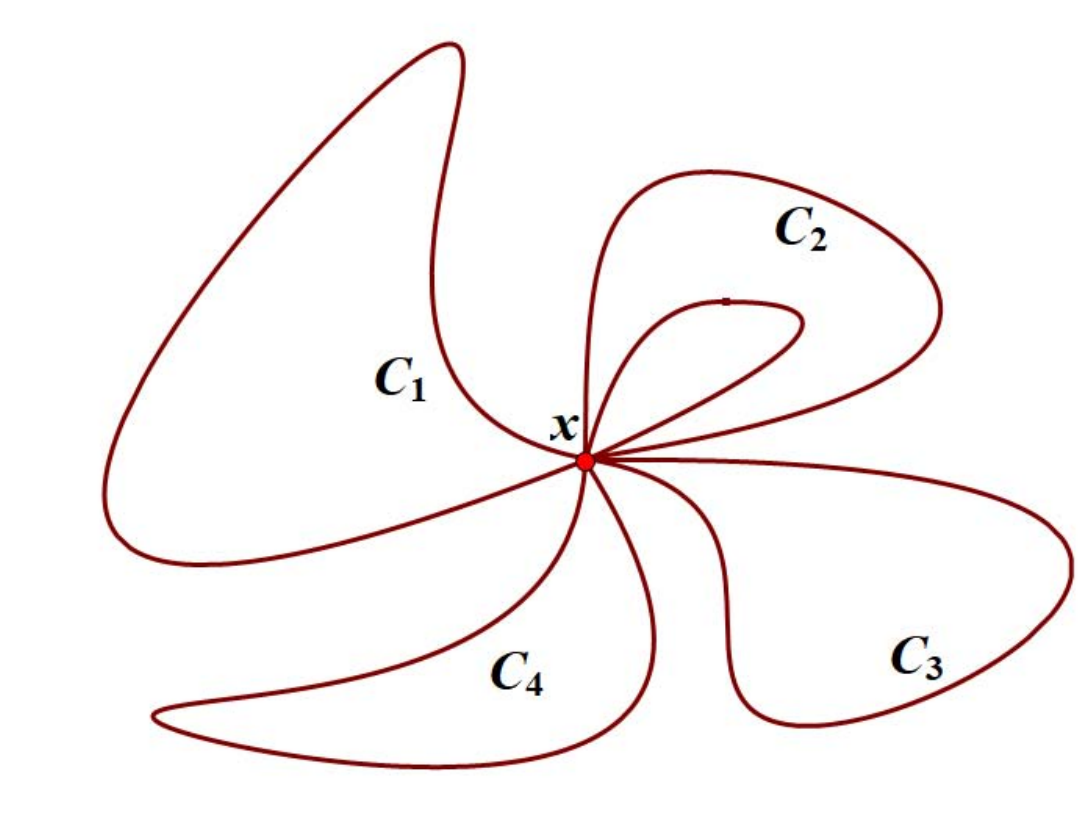}
\caption{ A cut point $x$ and  four pieces $C_1,C_2,C_3,C_4$.}
   \label{fig4p}           
\end{figure}

\begin{thm}\label{thm23}
The continuum $C\subset \mathds{R}^2$ is staircase connected, if and only if

$(1)$ it is orthogonally convex, and

$(2)$ for any point $x\in C$, $\mathrm{hw}_{x}(C)\neq0 $ or $ \mathrm{vw}_{x}(C)\neq0$, and the same holds for any piece of $C$ instead of $C$.
\end{thm}

\begin{proof}
Suppose $C\subset \mathds{R}^2$ is staircase connected.
Then $C$ is orthogonally connected and orthogonally convex, by Lemma \ref{lem0}.
Consequently, for any point $x\in C$ and any other point $y\in C$, there is an orthogonal path in $C$ joining $x$ and $y$. Hence, $\mathrm{hw}_{x}(C)\neq0 $ or $ \mathrm{vw}_{x}(C)\neq0$. Analogously, the same holds for each piece.

Next, we prove the sufficiency. Let $a,b$ be any two points of $C$.

Suppose without loss of generality that the coordinates of $b$ are not smaller than those of $a$. We start at $a$ an orthogonal path $P$. It can be started, because $\mathrm{hw}_a( C ) \neq{0} $ or $\mathrm{vw}_a( C ) \neq{0} $.
We wish to come closer to $b$, horizontally or vertically.
If this is impossible, then there exist a horizontal line-segment $au$ and a vertical one $av$, with $au\backslash\left\{a\right\}$ and $av\backslash\left\{a\right\}$ disjoint from $C$, such that $b\in \mathrm{conv}([au\rangle \cup [av\rangle)$. See Figure~\ref{figjhau}. Since $S$ is orthogonally convex, $([au\rangle \cup [av\rangle) \cap C=\left\{a\right\}$. Without loss of generality, let $\mathrm{hw}_a( C ) \neq{0} $ and $am\subset C$, where $m$, $a$,  and $u$ are collinear.
Thus, $[au\rangle \cup [av\rangle$ separates $b$ from the point $m$.
This implies that $a$ is a cut point of $C$. But the piece determined by $a$ containing $b$ does not fulfil the requirement (2) for pieces.
\begin{figure}[htbp]
\begin{minipage}[t]{0.4\linewidth}
\centering
\includegraphics[width=0.95\textwidth]{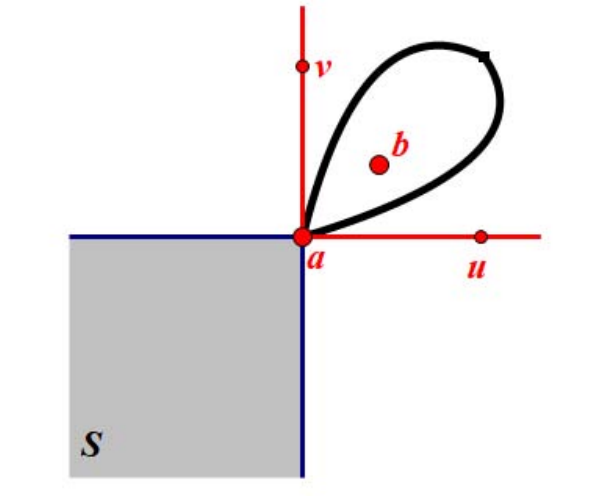}
\caption{$b\in \mathrm{conv}([au\rangle \cup [av\rangle)$.}
   \label{figjhau}           
\end{minipage}%
\hfil
\begin{minipage}[t]{0.4\linewidth}
\centering
\includegraphics[width=1\textwidth]{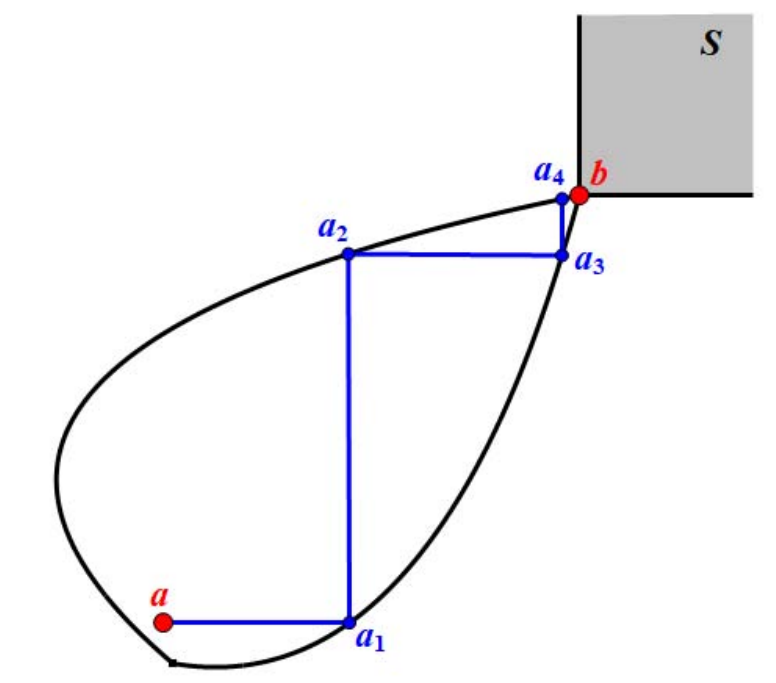}
\caption{$c$ is a cut point of $C$.}
   \label{figjh22}
\end{minipage}
\end{figure}

Hence, starting in  $a$,  we can come closer to $b$, say horizontally, until
we reach the vertical line $Y_b$ or as long as we remain in $C$, up to some point $a_1$. If $a_1\in Y_b$, then $[a,a_1,b]$ is the desired staircase.
If $a_1\in (\mathrm{bd}C)\backslash L$, we repeat the same procedure with $a_1$ instead of $a$, and so on, until either we reach $b$, or we obtain a sequence of points $a_1,a_2,a_3,\cdot\cdot\cdot\in \mathrm{bd}C$.
In the second case, since $C$ is compact, $\left\{a_n\right\}^\infty_{n=1}$ converges to some point $c\in C$. Obviously, $\|a_n-a_{n+1}\|\rightarrow0$.
If $c$ is not a cut point of $C$, then the orthogonal convexity of $C$ yields $\mathrm{hw}_c( C ) = \mathrm{vw}_c( C )={0} $, contradicting the hypothesis. In case $c$ is a cut point of $C$, it only follows that $\mathrm{hw}_c( P ) = \mathrm{vw}_c( P )={0} $ for the piece $P\ni a$ determined by $c$, again contradicting the requirement (2) for pieces. See Figure~\ref{figjh22}.
\end{proof}

\section{Staircase connectedness and convex bodies}

Let $K\subset \mathds{R}^2$ be a convex body and $x\in \mathrm{bd}K$. Let
\begin{align}
  T_x(K)=\bigcup\limits_{ y\in K \backslash\left\{ x\right\} }  [xy\rangle
\end{align}
and
\begin{align}
  \alpha_x(K)=\mu_1\left\{\frac{y-x}{\|y-x\|}:y\in K \backslash \left\{x\right\}\right\}
\end{align}

We say that $K$ is $obtuse$, if, for any $x\in \mathrm{bd}K$, either $\alpha_x(K)> \pi/2$ or, simultaneously, $\alpha_x(K)=\pi/2$ and $T_x(K)\backslash \left\{x\right\}$ is not open.
\begin{thm}\label{thmobtuse}
Each obtuse planar convex body is staircase connected.
\end{thm}
\begin{proof}Let $C$ be an obtuse planar convex body. Being convex,  $C$ is orthogonally convex.
For any $a\in \mathrm{bd}C$, if $\alpha_a(C)>\pi/2$, at least one of $\mathrm{hw}_a(C)$ and $\mathrm{vw}_a(C)$ is non-zero.
If $\alpha_a(C)=\pi/2$ and $T_a(C)\backslash \left\{a\right\}$ is not open,
then $\mathrm{bd} T_a(C)$ includes a non-degenerate line-segment, whence $\mathrm{cl}T_a(C)$ includes a horizontal or vertical such line-segment, and therefore at least one of $\mathrm{hw}_a(C)$ and $\mathrm{vw}_a(C)$ is non-zero.
For any $a\in \mathrm{int}C$, it is obvious that $\mathrm{hw}_a(C)$ and $\mathrm{vw}_a(C)$ are non-zero.
It follows from Theorem \ref{thm21} that $C$ is staircase connected.
\end{proof}

For a convex body $K\subset \mathds{R}^2$, we call $a\in \mathrm{bd}K$ a $non$-$obtuse~vertex$, if $\alpha_a(K)\leq\pi/2$.
\begin{thm}\label{thmacute}
Suppose $a$ and $b$ are non-obtuse vertices of the convex body $K\subset \mathds{R}^2$. If some coordinate axis is parallel to $\overline{ab}$, then $K$ is staircase connected.
\end{thm}
\begin{proof}
Suppose that $\overline{ab}$ is parallel to the $x$-axis. Then, $ab\subset K$ and both $\mathrm{hw}_a(K)$ and $\mathrm{hw}_b(K)$ are positive. Also, $K$ is contained in a strip bounded by $Y_a$ and $Y_b$.
It is obvious that, for any point $q\in K\backslash\left\{a,b \right\}$, $\mathrm{vw}_q(K)\neq 0$.
In addition, $K$ is orthogonally convex, being convex.
Therefore, $K$ is staircase connected, by Theorem \ref{thm21}.
\end{proof}

\begin{thm}\label{thmrotate}
In $ \mathds{R}^2 $, every convex body can be rotated to become staircase connected.
\end{thm}
\begin{proof}Let $D$ be a convex body in $\mathds{R}^2$.
Since $D$ is convex, any  set obtained by rotating $D$ is still convex and therefore orthogonally convex.

If $D$ is obtuse, $D$ is staircase connected, by Theorem \ref{thmobtuse}.
Otherwise, choose $a,b\in \mathrm{bd}D$ such that $a$ is a non-obtuse vertex, and such that $b$ is non-obtuse if there exists a second such vertex.
We rotate the set $D$ to a position $D'$ to bring $ab$ parallel to some coordinate axis.
If  both $a$ and $b$ are non-obtuse, use Theorem \ref{thmacute}.
Otherwise, one of $\mathrm{hw}_x(D')$ and $\mathrm{vw}_x(D')$ is not zero for any point $x\in \mathrm{bd}D'$, and the conclusion follows from Theorem \ref{thm21}.
\end{proof}

Since the union of two orthogonally connected sets with non-empty intersection is orthogonally connected, it is easy to derive the following.
\begin{thm}\label{thmum}
Any connected union of finitely many obtuse convex bodies in $\mathds{R}^2$ is orthogonally connected.
\end{thm}

\begin{thm}\label{thmu2st}
The convex hull of a union of finitely many obtuse convex bodies in $\mathds{R}^2$ is staircase connected.
\end{thm}

\begin{proof}Let $A_1,...,A_n$ be obtuse convex bodies and $K$ the convex hull of their union.
For any $x\in \mathrm{bd}K$, either $x\in \mathrm{bd}A_i$ for some index $i$, or $\alpha_x(K)=\pi$.
It follows from the definition that $K$ is obtuse.
Thus, $K$ is staircase connected, by Theorem \ref{thmobtuse}.
\end{proof}

\section{Characterization via unimodal functions }

Let $A\subset \mathds{R}^2$ equal  the closure of its interior, which is  bounded and connected.
Let $L_x$ be the line of all points of abscissa $x$, i.e. $L_x=Y_{(x,0)}$.
There exist two numbers $a,b$ such that $L_x\cap A\neq \emptyset$ for all $x\in [a,b]$, the interval $[a,b]$ being maximal.

For any $x\in[a,b]$ define $f^+(x)=\mathrm{max}\left\{y:(x,y)\in A\right\}$ and
$f^-(x)$=$\mathrm{min}$\\$\left\{y:(x,y)\in A\right\}. $
Clearly, $f^+\geq f^-$.
Even $f^+(x)>f^-(x)$ for $x\in ]a,b[$.
The functions $f^+$ and $f^-$ are $associated$ to $A$.
If $f^+(a)=f^-(a)$ and $f^+(b)=f^-(b)$, we say that $A$ is $normal$.

A function $f:[a,b]\rightarrow \mathds{R}$ is called $unimodal$, if it is non-decreasing on a subinterval $[a,c]$ and non-increasing on $[c,b]$, for some $c\in]a,b[$.
Obviously, $f^+(a)$ and $f^-(a)$ are or are not equal.
What are the chances that they are not equal?
It is not difficult to see that $A$ can be rotated in the plane, such that $f^+(a)=f^-(a)$ and $f^+(b)=f^-(b)$.
In fact, only at most countably many rotations lead to inequality at $a$ or $b$.

Thus, the result below concentrates on the case of a normal set $A$.

\begin{thm}\label{thmunim}
Let $A\subset \mathds{R}^2$ be normal and vertically convex.
Then, $A$ is staircase connected, if and only if both $f^+$ and $-f^-$ are unimodal, where $f^+$ and $f^-$ are associated with $A$.
\end{thm}
\begin{proof}
Suppose both $f^+$ and $-f^-$ of $A$ are unimodal.
It is clear that $A$ is horizontally convex.
Since $A$ is also vertically convex by the hypothesis, it follows that $\mathrm{int}A$ is staircase connected.
Thus, it is enough to just consider the boundary points of $A$.

Denote $A_1=\left\{(x,f^+(x)): x\in ]a,b[\right\}$ and $A_2=\left\{(x,f^-(x)): x\in ]a,b[\right\}$. Let $u=(a,f^+(a))$ and $v=(b,f^+(b))$.
Then both $\mathrm{hw}_u(A)$ and $\mathrm{hw}_v(A)$ are non-zero because $f^+$ and $-f^-$ are unimodal.
For any point $p\in A_1\cup A_2$, we have $\mathrm{vw}_p(A)\neq 0$.
Let $A_3=\mathrm{bd}A\backslash (A_1\cup A_2\cup \left\{u,v\right\})$.
Possibly, $A_3\neq \emptyset$.
But for any point $p$ in $A_3$, we also have $\mathrm{vw}_p(A)\neq 0$.
Thus, the sufficiency holds, by Theorem \ref{thm21}.

Suppose now that $A$ is staircase connected.

Assume that at least one of $f^+$ and $-f^-$ is not unimodal, say $f^+$.
Then there are three points $p,q ,r\in[a,b]$, where $p<q<r$, such that $f^+(p)>f^+(q)$ and $f^+(q)< f^+(r)$.
But this contradicts the horizontal convexity of $A$.
\end{proof}

\section{$s$-extreme points}
A point of a staircase connected set $M$ will be called an $s$-$extreme~point$ of $M$, if it belongs to a staircase in $M$ only as an endpoint.

\begin{thm}\label{thmextp}
Let the staircase connected set $M\subset \mathds{R}^2$ be the closure of its  bounded and connected interior. The point $a\in M$ is an $s$-extreme point of $M$ if and only if either $\mathrm{hw}_a(M)=$ or $\mathrm{vw}_a(M)=0$.
\end{thm}
\begin{proof}
\begin{figure}[htbp]
\centering
\includegraphics[width=0.3\textwidth]{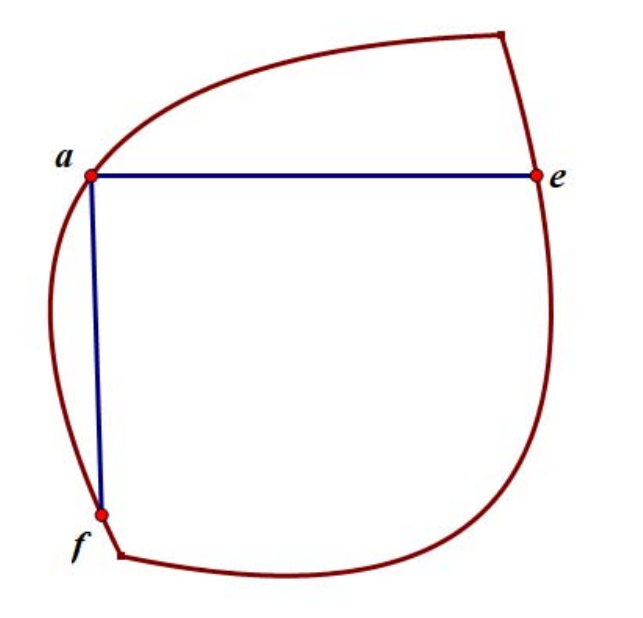}
\caption{}
\label{figaef}           
\end{figure}

By Theorem \ref{thm21}, $\mathrm{hw}_x(M)=\mathrm{vw}_x(M)= 0$ for no $x\in M$.
To prove the "only if" implication,
assume both $\mathrm{hw}_a(M)\neq 0$ and $\mathrm{vw}_a(M)\neq 0$.
Denote a horizontal line-segment in $M$ with endpoint $a$ by $ae$ and a vertical line-segment in $M$ with endpoint $a$ by $af$.
See Figure~\ref{figaef}.
Then $ae\cup af$ is a staircase, but $a$ is not an endpoint of it.
This contradicts the assumption that $a$ is an $s$-extreme point.

\begin{figure}[h]
\centering
\includegraphics[width=0.55\textwidth]{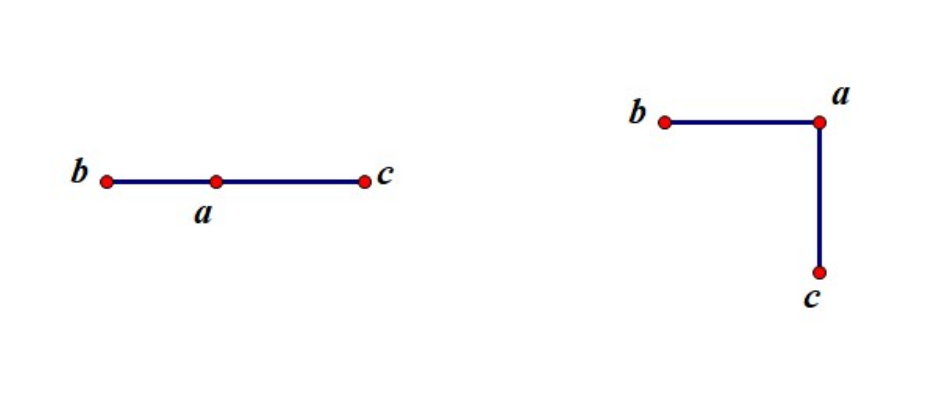}\\
\caption{}
\label{figbac3}           
\end{figure}

The ``if" part:
Assume that $a$ is not an \emph{s}-extreme point.
Then,
there exists a staircase $[b,a,c]$ in $M$.
There are two cases to consider:
$a,b,c$ are collinear or not. See Figure \ref{figbac3}.

In the first case, suppose, without lose of generality, that $[b,a,c]$ is a horizontal line segment.
We have $\mathrm{vw}_a(M)= 0$. 
An interior point of $M$ close to $b$ and an interior point of $M$ close to $c$ are separated by $Y_a$. Since $\mathrm{int}M$ is connected and $a\notin \mathrm{int}M,~Y_a\cap \mathrm{int}M$ contains a point $a'\neq a$. Since $M$ is, by Lemma \ref{lem0}, orthogonally convex, $aa'\subset M$, which contradicts $\mathrm{vw}_a(M)= 0$.

In the second case, $\mathrm{hw}_a(M)\neq 0$ and $\mathrm{vw}_a(M)\neq 0$, which contradicts the assumption.
\end{proof}

\section*{Acknowledgements}

This work is supported by NSF of China (12271139, 12201177, 11871192); the High-end Foreign Experts Recruitment Program of People's
Republic of China (G2023003003L); the Program for Foreign Experts of Hebei Province; the Hebei Natural Science
Foundation (A2024205012, A2023205045); the Special Project on Science and Technology Research and Development Platforms, Hebei Province (22567610H);
the Science and Technology
Project of Hebei Education Department (BJK2023092)
and the China Scholarship Council.

\noindent
$\mathrm{\underline{Xuemei~He},  \underline{Xiaotian~Song}}$\\
School of Mathematical Sciences\\[-2pt]
Hebei Normal University\\[-2pt]
050024 Shijiazhuang,
China\\
xuemei\_he20@163.com, Xiaotiansong163@163.com

\bigskip

\noindent
$\mathrm{\underline{Liping~Yuan}}$ (corresponding author)\\
School of Mathematical Sciences\\[-2pt]
Hebei International Joint Research Center for Mathematics and Interdisciplinary Science\\[-2pt]
Hebei Key Laboratory of Computational Mathematics and Applications\\[-2pt]
Hebei Normal University\\[-2pt]
050024 Shijiazhuang,
China\\
lpyuan@hebtu.edu.cn

\vspace{0.5cm}
\noindent
$\mathrm{\underline{Tudor~Zamfirescu}}$\\
School of Mathematical Sciences\\[-3pt]
Hebei International Joint Research Center for Mathematics and Interdisciplinary Science\\[-3pt]
Hebei Normal University\\[-3pt]
050024 Shijiazhuang, China\\
Fachbereich Mathematik\\[-3pt]
TU Dortmund\\[-3pt]
Dortmund,
Germany\\
Roumanian Academy\\[-3pt]
Bucharest,
Romania\\
tuzamfirescu@gmail.com

\end{document}